\documentclass[12pt,journal,onecolumn ]{IEEEtran}

\usepackage{enumerate}

\usepackage[utf8]{inputenc}
\usepackage[margin=1in]{geometry}
\usepackage{graphicx}
\usepackage{amsmath,amssymb,amsthm,authblk,mathrsfs,subfigure}
\usepackage[acronym]{glossaries}
\usepackage{color}
\usepackage[english]{babel}
\usepackage{comment}
\usepackage[nottoc]{tocbibind} 
\usepackage[outdir=./]{epstopdf}

\usepackage{url}
\usepackage{thmtools}

\usepackage{setspace}

\newcommand{\be}{\begin{equation}}
\newcommand{\ee}{\end{equation}}

\newcommand {\R}{\mathbb R}

\theoremstyle{plain}

\makeatletter
\newcommand{\sq}[1]{\mathbin{\mathpalette\make@circled{#1}}} 
\newcommand{\make@circled}[2]{%
	\ooalign{$\m@th#1\smallbigcirc{#1}$\cr\hidewidth$\m@th#1#2$\hidewidth\cr}%
}
\newcommand{\smallbigcirc}[1]{%
	\vcenter{\hbox{\scalebox{1.2}{$\m@th#1\square$}}}%
}
\makeatother

\declaretheorem[name={Example},qed={\lower-0.3ex\hbox{$\square$}} ] {Example}

\doublespace 

\title{On the gain of entrainment in the~$n$-dimensional   ribosome flow~model}
 \author{Ron Ofir and Thomas Kriecherbauer and Lars Gr\"{u}ne 
 and Michael Margaliot\thanks{RO is with the Andrew and Erna Viterbi Faculty of Electrical and Computers Eng., Technion---Israel Institute of Technology, Haifa 3200003, Israel.   TK and LG  are  with the
Mathematical Institute, University of Bayreuth, Germany.
MM (michaelm@tauex.tau.ac.il) is with the School of Elec. Eng. - Systems, Tel Aviv University, 69978, Israel.
This research   is partially supported by    research grants from the~DFG (GR 1569/24-1 and KR 1673/7-1) and the~ISF.}}
 
\newtheorem{theorem}{Theorem}

\newtheorem{definition}{Definition}
\newtheorem{corollary}[theorem]{Corollary}

\newtheorem{proposition}[theorem]{Proposition} 
\newtheorem{remark}{Remark}

\newacronym{rfm}{RFM}{Ribosome Flow Model}
\newacronym{pmp}{PMP}{Pontryagin Maximum Principle}
\newacronym{tas}{TASEP}{Totally Asymmetric Simple Exclusion Process}

\newcommand{\mean}[1]{\overline{ #1}}

\newcommand*\diff{\mathop{}\!\mathrm{d}}

\definecolor{darkgreen}{rgb}{0, 0.6, 0}

\begin{document}
	\maketitle 

\begin{abstract}
    The ribosome flow model (RFM) is a phenomenological  model for the flow of particles along a~1D chain of~$n$ sites. It has been extensively used to study ribosome flow  along the mRNA molecule during translation. When the transition rates along the chain are time-varying and jointly $T$-periodic the~RFM entrains, i.e., every trajectory of the RFM  converges to a unique $T$-periodic solution that depends on the transition rates, but not on the initial condition.
    In general, entrainment to periodic excitations like the 24h solar day or the 50Hz frequency of the electric grid is important in numerous  natural and artificial systems. 
    An interesting question, called the gain of entrainment~(GOE)   in the~RFM, is whether proper coordination of the periodic translation rates  along the mRNA can lead to a  larger average protein production rate.  Analyzing the~GOE in the RFM is non-trivial and partial results exist only for the~RFM with dimensions~$n=1,2$.  We use a new approach to derive  several results on the~GOE in the general $n$-dimensional~RFM. Perhaps surprisingly, we rigorously characterize  several cases where  there is no~GOE, so  to  maximize  the average production rate in these cases, the best choice  is to use constant transition rates 
    all along the chain. 
\end{abstract}


\section{Introduction}
 
The totally asymmetric 
simple exclusion process~(TASEP) is a
fundamental model in statistical mechanics~\cite{solvers_guide,kriecherbauer_krug2010,TASEP_book}.
The basic version of 
TASEP includes~$n$ sites, ordered along a 1D lattice, and particles that  hop randomly in a   unidirectional manner, that is, from site~$i$ to site~$i+1$.  Each site can be either empty or include a particle, 
and  a particle can only hop into an empty site. This generates an intricate coupling between the particles. 
In particular, if a particle is ``stuck'' in  a  site  for a long time then particles will accumulate in the sites behind it, thus generating a ``traffic jam'' of particles.

TASEP has been used to model the flow of ribosomes along the mRNA~\cite{TASEP_tutorial_2011} and other processes of   
  intracellular transport~\cite{RevModPhys.85.135}, pedestrian and vehicular traffic~\cite{TASEP_PEDES}, the movement of ants, and numerous other natural and artificial phenomena~\cite{TASEP_book}.  
However, rigorous analysis of TASEP is difficult and closed-form results are known  only for very 
special cases of transition rates along the lattice~\cite{solvers_guide}.

   The ribosome flow model~(RFM) is the mean-field approximation of TASEP~\cite{reuveni2011genome}. The RFM  is a deterministic model composed of~$n$ non-linear, first-order ODEs. 
  The RFM is highly amenable to analysis using tools from systems and control theory: the RFM  is a contracting system~\cite{sontag_cotraction_tutorial},  a cooperative system~\cite{margaliot2012stability},
  and, moreover, it is also a totally positive differential system~\cite{margaliot2019revisiting}. 
  The RFM can also be interpreted as a  compartmental  chemical reaction network~(CRN) 
  with transition rates that depend on the amount
of  particles 
and free space in various compartments~\cite{RFM_as_CRN}. The~RFM can also be derived via a special finite volume spatial discretization of widely used hyperbolic PDE
flow models~\cite{RFM_FROM_PDE}, and can also be represented as a port-Hamiltonian  system~\cite{RFM_AS_PORT}. 
 
  Understanding the regulation and dynamics of ribosome  flow and, consequently, protein production rate is a fundamental problem in biology and medicine. 
  For example, ``traffic jams'' of ribosomes along the mRNA 
  have been implicated with various diseases~\cite{neurojams,tuller_traffic_jams2018}. Rigorous analysis of ribosome flow
  is also important in scientific
  fields that include  manipulations of   
  the translation machinery like synthetic biology,
  mRNA viruses, and biotechnology.

  The RFM and its variants  has been extensively used to model and analyze the flow of ribosomes along the mRNA during translation (see, e.g.,~\cite{rfm_max,rfm_sense,EYAL_RFMD1,alexander2017,rfmr_2015,down_reg_mrna,rfm_feedback,randon_rfm,Aditi_abortions,Aditi_extended_2022,Ortho_RFM}).
  More recently, networks of interconnected RFMs have been used to model and analyze large-scale translation in the cell, and the effect of   competition for shared resources like ribosomes and tRNA molecules~\cite{allgower_RFM,Raveh2016,nani,aditi_networks,fierce_compete}. 
 Models based on networks of~RFMs
have  also been   validated experimentally. It was shown that such a model can  predict the density of ribosomes along different mRNAs, the protein levels of different genes, and can even be used for engineering ribosomal traffic jams (see, for example, \cite{Zur2020,reuveni2011genome,HALTER2017267}).  
  
  Ref.~\cite{RFM_entrain} studied the RFM when the entry, exit, and elongation rates along the chain are jointly periodic, with a common period~$T$, and proved  that in this case the dynamics admits a unique 
  $T$-periodic solution that is globally exponentially stable. 
  If we view the   rates as a periodic excitation (e.g., due to the periodic cell-cycle process) then the 
  state-variables in the RFM, and thus also the protein production rate, \emph{entrain} to the excitation. 
  
  The cell-cycle is a periodic program that controls  DNA synthesis and cell division.
   Proper execution
of the cell-cycle program, and its coordination with cell growth, entails  the expression and activation of key proteins at specific times along the period~\cite{translation_and_division}.  
  The eukaryotic cell cycle is controlled by periodic gene expression~\cite{peri_cell_cycle}. 
  There is growing evidence that protein levels of cell-cycle related genes are  regulated also  via the   translation machinery. 
Ref.~\cite{Higareda2010} reports that the expression of the 
human translation initiation factor 3~(eIF3)    oscillates during cell-cycle, with one
maximum expression peak in the early S phase and a second during mitosis. Ref.~\cite{Frenkel2012}   argue that periodicity in the tRNA levels of
bottleneck  codons induces periodicity in the translation rate. 
Ref.~\cite{patil2012}  showed that the levels of 16 tRNA modifications, and thus the translation efficiencies of different codons oscillate during cell-cycle.   Their results imply that translation regulation has a direct role in cell-cycle related oscillations.   
  Another possible regulation mechanism is via control of
the transcription rate of tRNA genes (and other genes), resulting in oscillations in intra-cellular tRNA
levels [12]. Since the decoding time of codons is affected by the available levels of the tRNA molecules
recognizing them, this may eventually lead to oscillations in the decoding times
of different codons.
 A recent paper~\cite{cell_cycle_dep_translation} used a sophisticated  single-cell ribo-seq method
to show that 
limitation for a particular amino acid causes ribosome pausing at a subset of the codons encoding the amino acid. However, this pausing is only observed in a sub-population of cells correlating to its cell cycle state.

  Refs.~\cite{RFM_IEEE_CL,max_period_output_RFM} considered    the gain of entrainment~(GOE) in the~RFM.
  To explain the GOE problem, consider a general  nonlinear
  control system:
  \begin{align}\label{eq:sys}
  \dot x(t)&=f(x(t),u(t) ),\nonumber \\
  y(t)&=g(x(t)), 
  \end{align}
  where~$x\in\R^n$ is the state-vector,~$u\in\R^m$ is the input,   and~$y\in\R^1$ is the scalar  output. We assume that  the output represents  
  a quantity that should be maximized (in some suitable sense). For example, in the context of biotechnology, the output may be the     production rate of a desired protein. 
  
  A control~$u$ is called~$T$-periodic if~$u(t+T)=u(t)$ for all~$t\geq 0$.
  Suppose that for any~$T$-periodic control the system~\eqref{eq:sys}
  entrains, that is, $x$  converges to a unique $T$-periodic solution~$\gamma(t)$, $t\in[0,T)$, that   depends on the control, but not on the initial condition~$x(0)$. Then the output converges to the    scalar~$g(\gamma)$. 
  Now fix a $T$-periodic control~$u$, and let
  \begin{align*}
      \mean{u} &:=\frac{1}{T}\int_0^T u(t)\diff t, \\
      \mean{y} &:=\frac{1}{T}\int_0^T g(\gamma(t)) \diff t,  
      \end{align*}
      that is, the averages of the input and the  unique periodic    output.
  If we apply the constant 
  control~$v(t):=\mean{u}$ then the entrainment property implies  that the state will converge to a unique  equilibrium~$e$, 
  that   depends on~$\mean{u}$, but not on the initial condition~$x(0)$, and thus~$y(t)$ converges to~$g(e)$. We say that the system admits a GOE for the input~$u$ if
  \[
            \mean {y} > g(e).  
  \]
 In other words, we  consider 
 two controls~$u(t)$ and~$v(t):=\mean{u} $ with the same average,  
  and compare the average of the corresponding asymptotic  outputs. A GOE implies that entrainment  does not only entail synchronization to periodic excitations, but also an \emph{improvement in the average output}.  This 
 general  
  property  may be  relevant in many fields including 
  protein production during the cell cycle division program, 
  seasonal cycles of infectious diseases~\cite{seasonal_epidemics}, 
   periodic fishery~\cite{peri_fishing}, 
  and  periodically-operated chemical reactors~\cite{Periodic_Chemical_Reactors}. 
  
 Since periodic controls allow more flexibility than constant controls, one may intuitively expect that many systems admit GOE.  
  However,  analysing the GOE in a nonlinear system is non-trivial. Many non-linear systems do not entrain at all,  that is, the response to a periodic input is not necessarily convergence  to   a periodic solution (see, e.g.,~\cite{NIKOLAEV20181232}).  
  Even if the system does entrain,  the periodic solution~$\gamma(t)$ (and its integral along a period) is
  typically not known explicitly.

  A dynamical system is called \emph{contractive} if any two solutions approach each other at an exponential rate~\cite{bullo_contractive_systems,sontag_cotraction_tutorial}.
 Contractive systems entrain to periodic excitations~\cite{sontag_cotraction_tutorial,LOHMILLER1998683} and also admit a well-defined
  frequency response function~\cite{pavlov}, but again this is not   known explicitly. Thus, rigorous analysis of the GOE   
 is not immediate  even for scalar contractive systems. 
  The next example demonstrates   this in  a simple synthetic system. 
  \begin{Example}\label{exa:tanh}
  Consider the scalar   control system
  \begin{align*}
      \dot x(t) &= 1-\frac{3}{2} \tanh(x(t))+u(t),\\ 
  y(t)&=x(t).
  \end{align*}
 Note that this  is in  the form in~\eqref{eq:sys} 
  with~$f(x,u)=1-\frac{3}{2} \tanh(x )+u$ and~$g(x)=x$.
    Assume that~$u(t)$ takes values in~$[-1,1]$. The state space of the dynamics is~$\Omega:=\R_+$. 
  The Jacobian of the vector field is~$J(x)= \frac{-3}{2\cosh^2(x)} $, so in particular~$J(x)<0$ for all~$x\in\Omega$ implying that the system is contractive and thus entrains to periodic excitations. 
  We consider two controls. The first is~$u(t)=\sin(2\pi t)$ that is~$T$-periodic, with~$T=1$,  and has average~$\mean u=0$, and the second is the constant control~$v(t)\equiv \mean u  = 0$. 
  For the latter control the solution converges to the equilibrium point~$e:= \tanh^{-1}(2/3)\approx     0.8047
 $, so the output converges to~$ g(e)= e$. Fig.~\ref{fig:tanh} depicts the attracting periodic solution~$\gamma(t)$, $t\in[0,1]$, for the control~$u(t)$. Numerically computing\footnote{All the
simulations in this paper  were performed using Matlab, and  all numerical values are to four digit accuracy.} the average of this solution over a period yields~$\mean y=\int_0^1 \gamma(t)\diff t =0.8127 $, and thus the system admits GOE for this control.
  \end{Example}
  
  \begin{figure}
 \begin{center}
  \includegraphics[scale=0.6]{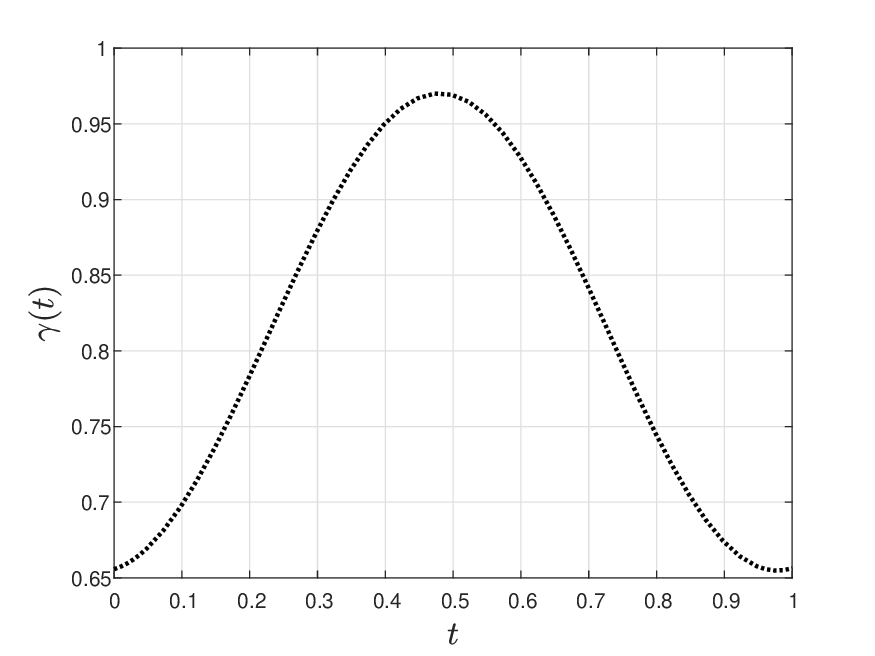}
	\caption{Periodic solution~$\gamma(t)$ as a function of~$t$ for the system in Example~\ref{exa:tanh}.}
	\label{fig:tanh} 
\end{center}
\end{figure}
  
  Analyzing GOE is closely related to 
     periodic optimal control~\cite{peri_opti_cont}.
    The basic idea is to pose the problem of 
    maximizing the average of the  output, \emph{along the periodic solution},  under an integral constraint on the control, namely,~$\frac{1}{T}\int_0^T u(t)\diff t = \mean{u}$. 
   However,   it seems that
   the application of 
   standard analysis tools like Pontryagin's maximum principle  to this problem 
   provide explicit information only for low-dimensional   systems, e.g., when~$n=1$~\cite{1dGOE,scalar_int}. 
   Another known approach for proving the optimality of equilibrium solutions is based on
   passivity/dissipativity techniques~\cite{GruM16}. However, so far for the~RFM we have not been able to establish the kind of passivity/dissipativity that we would need for this kind of analysis. While dissipativity results for generalized RFMs exist \cite{RFM_AS_PORT}, it is not clear how they can be extended such that the transition rates play the role of the control, which would be needed for a GOE analysis with respect to periodic variations of these rates. Moreover, the integral constraints on the control cause  technical difficulties. 
   
 Protein synthesis is known 
   to be one of the most energy consuming processes in the cell~\cite{energy_consuming}. A  GOE may make this process more efficient on average, and thus reduce the effective cost.
  To study  this problem in the context of translation,   we develop a new approach to analyzing   GOE in  the~$n$-dimensional~RFM. This is based on integrating the differential equations and ``higher-order moments'' of these differential equations along the periodic solution. 
  
  It turns out that the algebraic derivations are simplified if we actually consider  the differential equations for the  centered variables obtained by  subtracting from each state  variable its equilibrium point corresponding to 
   the case of  constant controls~$u_i(t)\equiv \mean{u}_i$ for all~$i$. Using this, we prove several new results (see Theorem~\ref{thm:main} below). For example, we show that for an~$n$-dimensional RFM, with~$n$ odd, and equal 
   internal rates   there 
   is no~GOE. 
   To the best of our knowledge, these are the first analysis results on the~GOE in a general  $n$-dimensional~RFM.  
  
  The remainder of this paper is organized as follows.  The next section reviews the RFM and formally defines the GOE in the~RFM. 
  Section~\ref{sec:main} presents our main results. These show that in various~$n$-dimensional RFMs there is no~GOE. 
   The final section concludes the paper.

  \section{Preliminaries}
  In this section, we briefly review the RFM, and the problem of  analyzing~GOE. 
  \subsection{The Ribosome Flow Model}
   The RFM 
   is a phenomenological 
   compartmental 
model that includes~$n$ sites ordered along a 1D chain. Particles flow along the chain from left to right (i.e., from site~$i$ to site~$i+1$). 
The~RFM includes~$n$ state variables~$x_i(t)$, $i=1,\dots,n$.
Every state variable takes values in~$[0,1]$, where~$x_i(t)$ represents the normalized density of particles at site~$i$ at time~$t$. Thus,~$x_i(t) =0 $
[$x_i(t)=1$] implies that site~$i$ is empty [completely full] at time~$t$. 
The state-space of the RFM is  thus  the~$n$-dimensional unit cube~$[0,1]^n$.

The~RFM also includes $n$ positive  time-varying  transition rates,~$u_0(t),\dots,u_n(t)$,
where a large value of~$u_i(t)$ implies that it is relatively easy for a particle to move from site~$i$ to site~$i+1$ at time~$t$. 
In particular~$u_0(t)$ [$u_n(t)$] 
is called the initiation rate [exit rate] at time~$t$. 
In the context of translation, every site corresponds to a group of 
codons along the mRNA, 
and~$u_i(t)$ models biophysical properties  like the abundance of cognate tRNA molecules at time~$t$.

The  RFM  consists of~$n$ deterministic   first-order ODEs:
\begin{equation}\label{eq:RFM_nd}
    \dot x_i(t) = u_{i-1}(t)x_{i-1}(t) (1 - x_{i}(t))   - u_i(t) x_i(t) (1 - x_{i+1}(t))  , \quad i = 1, \dots, n,
\end{equation}
where $x_0(t) := 1$ and $x_{n+1} (t):= 0$. If we view the transition rates as controls, then this is a nonlinear  control system,  as the equations include products of the state-variables and the controls.

For example, the RFM with~$n=3$ is given by
\begin{align}\label{eq:rfm3}
    \dot x_1& = u_{0} (1 - x_1)   - u_1 x_1 (1 - x_2),\nonumber\\
    \dot x_2&= u_1 x_1 (1 - x_2)-u_2x_2(1-x_3),\\
    \dot x_3&= u_2 x_2 (1 - x_3)-u_3 x_3,\nonumber
\end{align}
where for the sake of simplicity we omit specifying the dependence on~$t$.

To explain these equations, consider the equation for~$\dot x_2$ in~\eqref{eq:rfm3}.
The  term~$u_1 x_1(1-x_2)$
is the flow from site~$1$ to site~$2$. This is proportional to: the transition rate~$u_1$ from site~$1$ to site~$2$; 
the density~$x_1$ of particles in site~$1$; and  the ``free space'' $1-x_2$ in site~$2$.
Similarly, 
the  term~$u_2 x_2(1-x_3)$
is the flow from site~$2$ to site~$3$. Thus, the equation for~$\dot x_2$ states that the change in density in site~$2$ is the flow into site~$2$ minus the flow out of site~$2$.

The output  rate from site~$n$ is~$R(t):= u_n(t)x_n(t)$. When modeling translation, this  corresponds to ribosomes exiting the mRNA, and thus to the protein production rate. In the context of biotechnology, a natural goal is to maximize this production rate. More generally, 
in many transportation systems modeled using the RFM a natural goal is to maximize~$R(t)$ (in some well-defined sense).  

Just like TASEP, 
the RFM allows to model and analyze the evolution of traffic jams of particles. 
This is due to the fact that the entry rate to site~$i$ depends on the free space~$1-x_i$, which may be interpreted as a ``soft version'' of the simple exclusion principle. 
To explain this, assume that~$u_2(t)\equiv\varepsilon$, with~$\varepsilon$ positive and close to zero. 
Then
\[
\dot x_2 \approx     u_1 x_1 (1 - x_2) \geq 0,
\]
so we expect~$x_2$ to increase towards one. Then~$1-x_2$ decreases  towards  zero, so 
\[
\dot x_1 \approx    u_0  (1 - x_1) \geq 0,
\]
so~$x_1$ will also increase towards one. In this way, a traffic jam of particles
is formed behind the slow transition rate.

Consider 
positive and jointly $T$-periodic   transition rates~$u_0(t),\dots,u_n(t)$, 
 and denote the average of~$u_i(t)$ along a period by~$\mean{u}_i$.
Ref.~\cite{RFM_entrain} showed 
that in this case the RFM  admits  a unique globally exponentially stable~(GES) $T$-periodic solution~$\gamma$,
satisfying~$\gamma(t) \in (0,1)^n$ for all~$t\in[0,T)$. The next example demonstrates this.
\begin{Example}\label{exa:n3}
Consider  an  RFM with dimension~$n=3$,   positive and~$2\pi$-periodic rates
  $u_0(t)=3+\cos(t+5)$,
  $u_1(t)=1$,
  $u_2(t)=4+2\sin(t-4)$,
  $u_3(t)=2-\cos(t-1)$.
Fig.~\ref{fig:entra}
depicts the solutions~$x_i(t)$ for 
the   (arbitrarily chosen)   initial condition~$x(0)=\begin{bmatrix}
   0.3& 0.4 &0.5
  \end{bmatrix}^T$. It may be seen that every~$x_i(t)$ converges to a~$2\pi$-periodic solution~$\gamma_i$. The convergence to this~$\gamma$ holds for any initial condition~$x(0)\in[0,1]^3$. 
  \end{Example}

\begin{figure}
 \begin{center}
  \includegraphics[scale=0.6]{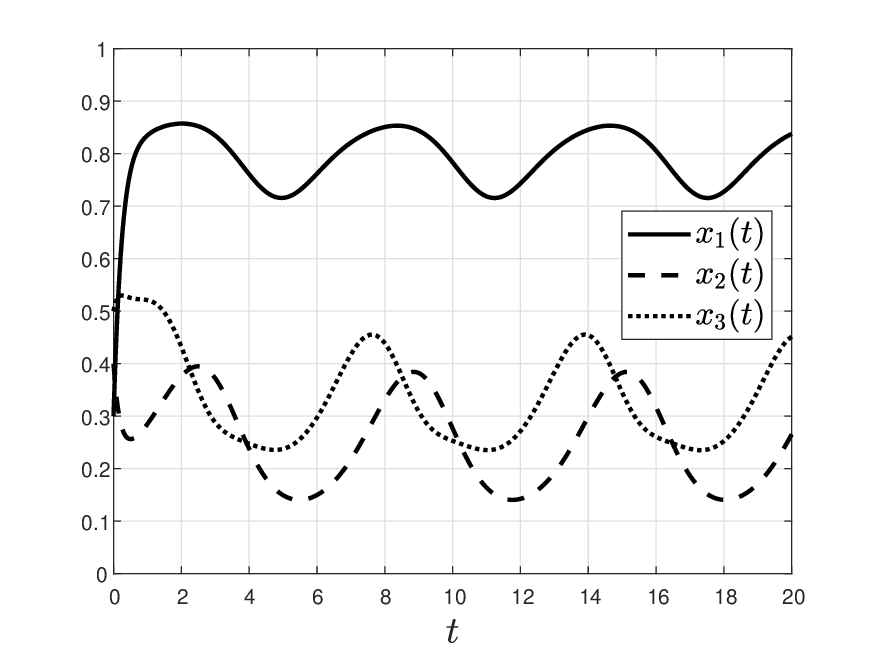}
	\caption{Trajectories~$x_i(t) $ as a function of~$t$ in
	an RFM  with~$n=3$ and $2\pi$-periodic rates. Every state-variable converges to a periodic pattern, with period~$2\pi$.}
	\label{fig:entra} 
\end{center}
\end{figure}

In particular, if the transition rates 
are set to   constant values~$u_i(t) \equiv \mean{u}_i$ for all~$i$ then the  RFM  admits a   unique GES equilibrium point~$e=\begin{bmatrix} e_1 & \dots & e_n \end{bmatrix}^T\in(0,1)^ n $.
It follows 
from~\eqref{eq:RFM_nd} that~$e$ satisfies
\begin{equation}\label{eq:RFM_nd_eq}
    \mean{u}_{i-1} e_{i-1} (1 - e_{i})   = \mean{u}_i e_i  (1 - e_{i+1})  , \quad i = 1, \dots, n,
\end{equation}
with $e_0 := 1$ and $e_{n+1} := 0$.

  \subsection{Gain of Entrainment in the RFM}
  Fix positive numbers~$\mean{u}_i$, $i=0,\dots,n$. We compare two cases. In the first, we apply the constant  control~$u_i(t)\equiv \mean{u}_i$, $i=1,\dots,n$. 
  Then the RFM state-variables converge to an equilibrium~$e:=\begin{bmatrix}
  e_1&\dots & e_n
  \end{bmatrix}^T\in(0,1)^n$ satisfying~\eqref{eq:RFM_nd_eq}. In particular, the protein production rate~$R(t) $ converges to
  \[
  R_C:= \mean{u}_ne_n,
  \]
  where the subscript~$C$ stands for constant.
  
  In the second case, we fix positive and~$T$-periodic controls~$u_i(t)$, such that~$\frac{1}{T}\int_0^T u_i(t)\diff t= \mean{u}_i$. This constraint implies that we invest, on average,  the same ``control effort'' as in the case of constant controls. 
  Then the~RFM state-variables converge to a unique $T$-periodic solution~$\gamma:[0,T)\to (0,1)^n$. 
  In particular, the average  
  production rate  converges to 
  \[
 R_P:=  \frac{1}{T} \int_0^T u_n(t) \gamma_n(t)\diff t, 
  \]
  where the subscript~$P$ stands for periodic. 
  
  We are interested in comparing~$R_C$ and~$R_P$. We begin with a simple example. 
\begin{Example}
We simulated  the RFM in Example~\ref{exa:n3} 
for the (arbitrarily chosen) initial 
condition~$ \begin{bmatrix}
   0.3& 0.4 &0.5
  \end{bmatrix}^T$. 
  The solution at time~$20\pi$ is
  \[
  x(20\pi)=\begin{bmatrix}
0.7809  &  0.1624&    0.3290
  \end{bmatrix}^T , 
  \]
 and we assume that this point is very close to a point on~$\gamma$. 
Now simulating the dynamics from~$x(0)=\begin{bmatrix}
 0.7809   & 0.1624 &   0.3290
  \end{bmatrix}^T$  yields 
 $x(2\pi)=\begin{bmatrix} 
  0.7809  &  0.1624&    0.3290
  \end{bmatrix}^T$,
  and
  \[
  R_{P} =\frac{1}{2\pi}\int _0^{2\pi }  u_3(t) x_3(t) \diff t = 0.5927. 
  \]
  The averages of the transition rates are
  $\mean {u}_0= 3$,
 $\mean {u}_1= 1$,
 $\mean {u}_2=4 $,
 $\mean {u}_3=2 $,
and~\eqref{eq:RFM_nd_eq} yields~$e_3=0.3085$, so~$R_C=\mean{u}_3e_3=0.6171$.
Thus, in this case~$R_C>R_P$. 
\end{Example}

We say that the RFM admits a GOE for a control~$u$ if~$R_P>R_C$. It is natural to expect  that the added flexibility in using $T$-periodic, rather  than constant controls, may yield  a~GOE. The difficulty in rigorously analyzing the GOE   is due
  to the fact that~$\gamma$, which is the periodic solution of an $n$-dimensional nonlinear dynamical system,  is not known explicitly. 
  Indeed, results that establish entrainment in various types of dynamical systems are typically based on implicit arguments like fixed point theorems~\cite{entrain_trans,margaliot2019revisiting,periodic_coop_1st_integral}, and do not provide explicit information on the 
  periodic solution.

  \section{Main Results}\label{sec:main}

 In this section, we analyze several scenarios where the $n$-dimensional RFM admits no GOE for any periodic control and any period~$T>0$.

 \begin{theorem}\label{thm:main}
  Fix an arbitrary~$T>0$.
  Consider the  $n$-dimensional  RFM with positive and jointly~$T$-periodic controls~$u_i(t)$, $i=0,\dots,n$. If any of the following  six conditions holds then there is no~GOE. 
  \begin{enumerate}[(I)]
      \item \label{item:nisodd}
      $n$ is odd, and for all $i \in \{1,3,\dots,n-2\}$ there exists $\alpha_i > 0$ such that
      \be\label{eq:allinterneqia}
      u_i(t)=\alpha_i u_{i+1}(t), \text{ for all } t\in[0,T) ,
      \ee
      i.e., all the internal rates along the chain are ``proportional in pairs''.
      \item \label{item:niseven_evenpairs}
      $n$ is even, and for all $i \in \{0,2,\dots,n-2$\} there exists $\alpha_i > 0$ such that
      \be\label{eq:paireql}
          u_i(t) = \alpha_i u_{i+1}(t), \text { for all } t\in[0,T).
      \ee
       \item \label{item:niseven_oddpairs}
      $n$ is even, and for all $i \in \{1,3,\dots,n-1$\} there exists $\alpha_i > 0$ such that
      \be\label{eq:paireql2}
          u_i(t) = \alpha_i u_{i+1}(t), \text { for all } t\in[0,T).
      \ee
      \item \label{item:constrates_butu0un}
      $n$ is odd, and for all $i \in \{1,\dots,n-1\}$,
      \[
        u_i(t) \equiv \mean{u}_i,
      \]
      i.e.,  all the transition rates  are constant, except for the 
      initiation  rate~$u_0(t)$ and the 
      exit rate~$u_n(t)$.
      \item \label{item:constrates_butun}
     $n$ is even, and  for all $i \in \{0,1,\dots,n-1\}$,
      \[
        u_i(t) \equiv \mean{u}_i,
      \]
      i.e.,  all the transition rates  are constant, except for the exit rate~$u_n(t)$.
      \item \label{item:constrates_butu0}
      $n$ is even, and for all $i \in \{  1,\dots,n\}$,
      \[
        u_i(t) \equiv \mean{u}_i,
      \]
      i.e.,  all the transition rates  are constant, except for the initiation  rate~$u_0(t)$.
  \end{enumerate}
  \end{theorem}

  Before proving this result, we discuss some of its implications. 
  In some applications, it may be possible to control~$u_0(t)$ and/or~$u_n(t)$, but not
  the other, ``internal'' transition rates.  For example, 
  it is often stated that control of translation in eukaryotic cells is exercised mainly at the initiation step, when ribosomes are recruited to mRNA~\cite{cell_cycle_dep,tuller_ramp,Hershey01122012}.  Specifically, assume
  that all the internal rates are constant, i.e. 
  \be\label{eq:allint}
  u_1(t)\equiv \mean{u_1},\dots, u_{n-1}(t)\equiv\mean{u}_{n-1},
  \ee
  and our goal is to determine $T$-periodic and positive rates~$u_0(t)$ and~$u_n(t)$, with fixed averages~$\mean{u}_0$ and~$\mean{u}_n$, 
  that   maximize the average output rate. If~$n$ is odd then  
  assertion~\ref{item:constrates_butu0un} in Theorem~\ref{thm:main} states that 
the constant rates~$u_0(t)\equiv\mean{u}_0$
and~$u_{n}(t)\equiv\mean{u}_{n}$ 
are an optimal choice. In the case of even~$n$ assertions~\ref{item:constrates_butun} and~\ref{item:constrates_butu0} in Theorem~\ref{thm:main} provide the weaker result that constant rates are optimal if we allow only for one of the two rates $u_0$ or $u_n$ to be time-dependent.

The distinction between the cases of an RFM with~$n$ odd and $n$ even is an artifact of our analysis approach. We conjecture that there is no GOE in the RFM for any~$n$ and any periodic controls, but we do not know how to prove this more general statement.

  The remainder of this section  is devoted to the proof of Theorem~\ref{thm:main}. 
  This requires several auxiliary results. We begin by defining the  ``centered'' 
  state-variables: 
  \begin{align}
z_i(t):=x_i(t)-e_i , \quad i=1,\dots,n. 
  \end{align}
  Then~\eqref{eq:RFM_nd} yields
  \begin{equation}\label{eq:RFM_nd_shift}
    \dot z_i(t) = u_{i-1}(t) (z_{i-1} (t)+ e_{i-1}) (1 - z_i (t)- e_i) - u_i(t) (z_i(t) + e_i)(1 - z_{i+1}(t) - e_{i+1}), 
    \quad i = 1,\dots,n, 
\end{equation}
with~$z_0(t) := 0$ and $z_{n+1} (t):= 0$. 

Note that as~$x  $ converges to the unique $T$-periodic solution~$\gamma$, $z $ converges to the unique  $T$-periodic solution~$\gamma -e$. Our analysis  is based on  computing  integrals of~$\frac{d}{dt} z_i(t) $ and~$\frac{d}{dt} z_i^2(t) $ \emph{along this  periodic solution}. To simplify the notation, we 
 write~$\int_0^T\frac{d}{dt} (z_i(t))^k \diff t$, $k\in\{1,2\}$, but we \emph{always} integrate along the unique $T$-periodic solution~$\gamma$.

It is useful to introduce  the following notation. For a $T$-periodic function~$y(t)$, let
\[
\mean{y}:=\frac{1}{T}\int_0^T y(t)\diff t,
\]
that is, the average of~$y(t)$ along a period. 
Note that~$\mean{z}_i=\mean{x}_i-e_i$, so~$\mean{z}_i>0$ if and only if~(iff) the  average of~$x_i$ for the periodic transition rates is larger than 
the equilibrium value~$e_i$ corresponding to the constant rates.
We also denote
\[
\eta_{i,j}:=\mean{u_i z_j},\quad \eta_{i,j,k}:=\mean{u_i z_j z_k }  , 
\]
where, as noted above, the averages 
are always computed along the unique  periodic solution. Recall that~$z_0(t)=z_{n+1}(t)\equiv 0$, so any 
  ``moment'' $\eta$ that includes~$z_0$ or~$z_{n+1}$  (e.g.,~$\eta_{i,j,n+1}=\mean{ u_{i} z_{j} z_{n+1} } $) is zero.

The next result uses the compartmental structure of the~RFM to  derive  a simple necessary and sufficient condition guaranteeing that there 
is no~GOE. 
\begin{proposition}\label{prop:suffu}
We have
\begin{equation}\label{eq:u0z1_unzn}
     \eta_{0,1} = - \eta_{n, n},
\end{equation}
and
\be\label{eq:pott}
  R_P=-\eta_{0,1} +R_C. 
\ee
In particular, 
if~$\eta_{0,1}\geq 0$ then there is no GOE, and if~$\eta_{0,1}>0$  then~$R_P$ is strictly smaller than~$R_C$.
\end{proposition}

\begin{remark}
Note that~$\eta_{0,1}=\mean{u_0z_1}$ may be interpreted as the ``average correlation''
between the initiation rate~$u_0$ and the centered density~$z_1=x_1-e_1$. Note also that~\eqref{eq:u0z1_unzn} implies that there exists a time~$\tau \in [0,T)$ such that
\[
u_0(\tau)z_1(\tau)=-u_n(\tau)z_n(\tau).
\]
\end{remark}

\begin{proof}
It follows from~\eqref{eq:RFM_nd_shift} that
\[
\dot z_1+\dots+\dot z_n = u_0(1-z_1-e_1) -u_n(z_n+e_n).
\]
Integrating this equation  gives
\[
0=\mean{u_0}(1-e_1)-\eta_{0,1} -\mean{u_n}e_n -\eta_{n,n},
\]
and using~\eqref{eq:RFM_nd_eq} proves~\eqref{eq:u0z1_unzn}. 
To complete the proof  note that
\begin{align*}
    R_P&=    \mean{ u_n \gamma_n}\\
    &= \mean{u_n ( z_n+e_n )}\\
    &=\eta_{n,n} +R_C,
\end{align*}
and applying~\eqref{eq:u0z1_unzn}
yields~\eqref{eq:pott}.
\end{proof}

The next result  is based on  integrating~$\dot z_i(t)$, $i=1,\dots,n$. Roughly speaking, this allows   to express  the ``third-order moment'' $ \eta_{i,i,i+1}=\mean{u_i z_i z_{i+1}}$ as 
a linear combination of three  ``second-order moments''. 
\begin{proposition}
For any    $i \in \{0,\dots,n\}$, we have
\begin{equation}\label{eq:etas}
    \eta_{i,i,i+1} = \eta_{0,1} + \eta_{i,i}(1 - e_{i+1}) - \eta_{i,i+1}e_i.
\end{equation}
\end{proposition}
   
\begin{proof}
  First, we prove~\eqref{eq:etas} for $i=0$ and $i=n$. For~$i=0$,~\eqref{eq:etas} becomes
  \begin{equation*}
      0 = \eta_{0,1} +0- \eta_{0,1},
  \end{equation*}
  which clearly holds. Similarly, for~$i=n$ we get
  \begin{equation*}
      0 = \eta_{0,1} + \eta_{n,n},
  \end{equation*}
  which holds due to Prop.~\ref{prop:suffu}. 
  
The rest of the proof assume that~$i \in \{1,\dots,n-1\}$. Integrating~\eqref{eq:RFM_nd_shift}  and  using~\eqref{eq:RFM_nd_eq}  gives
\begin{align}\label{eq:intzdot} 
   0 &= \mean{ u_{i-1} (1 - e_i - z_i) (z_{i-1} + e_{i-1})} - \mean{u_i (z_i + e_i)(1 - e_{i+1} - z_{i+1})} \nonumber   \\
     &= -\eta_{i-1,i-1,i} + \eta_{i-1,i-1} (1 - e_i) - \eta_{i-1,i} e_{i-1} + \eta_{i,i,i+1} - \eta_{i,i} (1 - e_{i+1}) + \eta_{i,i+1} e_i . 
\end{align}
For $i=1$
this becomes 
\begin{align}\label{eq:eta112}
    \eta_{1,1,2} &= \eta_{0,1} + \eta_{1,1}(1 - e_2) - \eta_{1,2}e_1 
\end{align}
(recall that  $z_0(t)  =z_{n+1}(t) \equiv 0$).
For~$i=2$, Eq.~\eqref{eq:intzdot} becomes
\begin{align*}
    \eta_{2,2,3} &= \eta_{1,1,2} - \eta_{1,1}(1 - e_2) + \eta_{1,2}e_1 + \eta_{2,2}(1 - e_3) - \eta_{2,3} e_2 ,
\end{align*}
and using~\eqref{eq:eta112} gives
\begin{align*}
    \eta_{2,2,3} &=   \eta_{0,1} + \eta_{2,2}(1 - e_3) - \eta_{2,3} e_2  . 
\end{align*}
Continuing in this manner 
completes  the proof.
  \end{proof}

The next result  is derived  by integrating~$\frac{1}{2} \frac{d}{dt} z_i^2(t) =   z_i(t) \dot z_i(t)$, $i=1,\dots,n$. This yields~$n$  lower bounds on~$\eta_{0,1}$. 
  
\begin{proposition}\label{prop:eta01ound}
  For any~$i=1,\dots, n$, we have
  \begin{align}\label{eq:eta01_ub_moments}
    \eta_{0,1} \geq 
    \eta_{i,i+1} e_i^2 - \eta_{i-1,i-1}(1 - e_i)^2.
  \end{align}
  Furthermore, equality is attained iff $z_i(t) \equiv 0$, that is, iff~$x_i(t)\equiv e_i$.
  \end{proposition}
 \begin{proof}
  By~\eqref{eq:RFM_nd_shift}, 
\begin{align*}
    \dot z_i z_i &= u_{i-1} z_i  (1 - e_i - z_i) (z_{i-1} + e_{i-1}) - u_i z_i (z_i + e_i)(1 - e_{i+1} - z_{i+1}).
\end{align*}
Integrating this  and rearranging terms gives
\begin{align}\label{eq:tightzi}
    \mean{(u_{i-1}(z_{i-1} + e_{i-1}) + u_i(1 - e_{i+1} - z_{i+1})) z_i^2} &= \mean{ u_{i-1} z_i (1 - e_i) (z_{i-1} + e_{i-1})} \nonumber \\&- \mean{u_i z_i e_i(1 - e_{i+1} - z_{i+1})}.
\end{align}
The left-hand side of this equation is non-negative because the controls~$u_i(t)$ are positive for all~$t$,  and along the periodic solution~$z_i(t)+e_i=\gamma_i(t) \in(0,1)$ for all~$t$. Furthermore, the left-hand side is zero iff $z_i(t) \equiv 0$. We conclude that
\begin{align}\label{eq:etaii}  
    0 &\le \eta_{i-1,i-1,i} (1 - e_i) + \eta_{i-1,i} (1 - e_i) e_{i-1} + \eta_{i,i,i+1} e_i - \eta_{i,i} (1 - e_{i+1}) e_i,
\end{align}
with equality iff $z_i(t) \equiv 0$, and substituting the expressions for the third-order moments~\eqref{eq:etas} completes the proof. 
 \end{proof}

Note that for $i=1$ and~$i=n$,  
Proposition~\ref{prop:eta01ound}
yields
\be\label{eq:firine2}
    \eta_{0,1} \geq \eta_{1,2}e_1^2  ,
\ee
and
\be\label{eq:secine2}
  \eta_{0,1} \geq   -\eta_{n-1,n-1}(1-e_n)^2  , 
\ee
respectively. Combining  this with Proposition~\ref{prop:suffu}
implies that there is no GOE if~$\eta_{1,2}\geq 0$ or if~$\eta_{n-1,n-1 }\leq 0$.

For~$n=1$, $z_2(t)\equiv 0$, so $\eta_{1,2}=\mean{u_1z_2}=0$ and~\eqref{eq:firine2} becomes~$\eta_{0,1}\geq 0$. Combining this with Prop.~\ref{prop:suffu} yields the following result.
\begin{corollary}
  In the RFM with~$n=1$ there is no GOE. Furthermore, $R_P = R_C$ iff
  \begin{equation}\label{eq:1d_strict_loe_cond}
      u_1(t) = \frac{\mean{u}_1}{\mean{u}_0} u_0(t), \text{ for all } t\in[0,T).
  \end{equation}
\end{corollary}
\begin{proof}
    The proof that $R_P \le R_C$ follows immediately from Prop.~\ref{prop:suffu}. Now assume that~$R_p=R_C$. Then~$z_1(t) \equiv 0$ along the unique periodic solution, so
    \[
        x_1(t) \equiv e_1 = \frac{\mean{u}_0}{\mean{u}_0 + \mean{u}_1}.
    \]
    Substituting  this in~\eqref{eq:RFM_nd} yields
    \[
        0 = u_0(t) (1 - e_1) - u_1(t) e_1 = \frac{\mean{u}_1 u_0(t) - \mean{u}_0 u_1(t)}{\mean{u}_0 + \mean{u}_1},
    \]
    and this completes the proof.
\end{proof}

The fact that in the 1D RFM there is no GOE was already proved in~\cite{max_period_output_RFM} using a more complicated argument (see also~\cite{katriel20,RFM_IEEE_CL}).
  
We can now prove Theorem~\ref{thm:main} for~$n\geq 2$. To prove assertion~\ref{item:nisodd}, suppose that~$n$ is odd, i.e.,~$n=2k+1$, and that
all the internal rates are proportional in pairs, that is, for any~$i \in \{1,3,\dots,n-2\}$ there exists $\alpha_i > 0$ such that
  \be \label{eq:allintewq}
    u_i(t)=\alpha_i u_{i+1}(t), \text{ for all } t\in[0,T).
  \ee
For~$i=1$, Eq.~\eqref{eq:eta01_ub_moments} becomes 
  \[
    \eta_{0,1} \geq 
    \eta_{1,2} e_1^2 = \alpha_1 \eta_{2,2}e_1^2  ,
  \]
so if~$\eta_{2,2}\geq 0$ then there is no GOE and the proof is completed. We may thus assume that~$\eta_{2,2} <0$. 
For~$i=3$, Eq.~\eqref{eq:eta01_ub_moments} becomes 
  \[
    \eta_{0,1} \geq 
    \eta_{3,4} e_3^2 -\eta_{2,2}(1-e_3)^2  ,
  \]
and using~\eqref{eq:allintewq} gives~$ \eta_{0,1} \geq 
    \alpha_3 \eta_{4,4} e_3^2 -\eta_{2,2}(1-e_3)^2 $.
Thus, if~$\eta_{4,4}\geq 0 $ then there is no GOE, and we may assume that~$\eta_{4,4} <0$.
Continuing in this manner, we may assume that~$\eta_{2j,2j}<0$, for all~$j\leq k$. 
For~$i=n=2k+1$, Eq.~\eqref{eq:eta01_ub_moments} becomes 
  \[
    \eta_{0,1} \geq 
    -\eta_{2k,2k}(1-e_{2k+1})^2 > 0 ,
  \]
and this completes the proof of the assertion~\ref{item:nisodd} in Theorem~\ref{thm:main}.

To prove assertion~\ref{item:niseven_evenpairs}, suppose that~$n$ is even and that~\eqref{eq:paireql} holds. Then~$\eta_{0,1}=\alpha_0\eta_{1,1}$, and we have established no GOE if $\eta_{1,1} \ge 0$. Therefore
we may assume that $\eta_{1,1} < 0$. Considering~\eqref{eq:eta01_ub_moments} with $i=2$ we conclude that there is no GOE if $\eta_{2,3} =\alpha_2\eta_{3,3}\ge 0$, and we may assume that~$\eta_{3,3} < 0$. Continuing in this manner, we have that $\eta_{n-1,n-1} < 0$, and we already know from inequality~\eqref{eq:secine2} that this implies that there is no GOE.  This completes the proof of assertion~\ref{item:niseven_evenpairs} in Theorem~\ref{thm:main}.

To prove assertion~\ref{item:niseven_oddpairs}, suppose that~$n$ is even and that~\eqref{eq:paireql2} holds. Then~$\eta_{n-1,n}=\alpha_{n-1}\eta_{n,n}$, and due to Proposition~\ref{prop:suffu} there is no GOE if $\eta_{n-1,n} \le 0$. Therefore
we may assume that $\eta_{n-1,n} > 0$. Considering~\eqref{eq:eta01_ub_moments} with $i=n-1$ we conclude that there is no GOE if $\eta_{n-2,n-2} =\alpha_{n-3}^{-1}\eta_{n-3,n-2}\le 0$, and we may assume that $\eta_{n-3,n-2} > 0$. Continuing in this manner, we have that $\eta_{1,2} > 0$, and we already know from inequality~\eqref{eq:firine2} that this implies that there is no GOE.  This completes the proof of assertion~\ref{item:niseven_oddpairs} in Theorem~\ref{thm:main}.

Assertions~\ref{item:constrates_butu0un}, \ref{item:constrates_butun}, and~\ref{item:constrates_butu0} follow immediately from the assertions~\ref{item:nisodd}, \ref{item:niseven_evenpairs}, and~\ref{item:niseven_oddpairs}, respectively, as one may choose $\alpha_i = \mean{u}_i/\mean{u}_{i+1}$ for the corresponding values of $i$.
This completes the proof of Theorem~\ref{thm:main}.

Reconsidering the proof of 
 Proposition~\ref{prop:eta01ound}
 allows to slightly strengthen Theorem~\ref{thm:main}. To do this, we introduce another definition. 
 \begin{definition}
  A $T$-periodic solution~$\gamma:[0,T)\to(0,1)^n$ of the~RFM is called  \emph{degenerate} if there exists an index~$i\in\{1,\dots,n\}$ such that~$\gamma_i(t)$ is constant. 
 Otherwise, $\gamma$ is called \emph{non-degenerate}. 
  \end{definition}
  
  \begin{corollary}
    Suppose that one of the cases  in Theorem~\ref{thm:main} holds. Then for any $T$-periodic control~$u(t)$ that yields a  non-degenerate~$T$-periodic solution we have
\be\label{eq:ineqrp}
R_P<R_C.
\ee
  \end{corollary}
In other  words, in this case constant rates are in fact strictly better than periodic ones. 
\begin{proof}
    Fix an arbitrary~$i\in\{1,\dots,n\}$. Since the periodic solution is non-degenerate, $z_i(t)=x_i(t)-e_i$ 
    is not zero at some time~$t_i\in[0,T) $, and this implies that the term on the left-hand side of~\eqref{eq:tightzi}
    is positive. 
    Thus, the inequality in~\eqref{eq:etaii}   becomes strict. Now  the arguments in the proof of Theorem~\ref{thm:main} yield~\eqref{eq:ineqrp}.
\end{proof}

While Theorem~\ref{thm:main} proves that there is no GOE in  certain cases, numerical experiments suggest that there is no GOE in general.
\begin{Example}\label{exa:random}
We simulated an RFM with~$n=4$.
Fig.~\ref{fig:4d_unxn} shows the moving average of the production rate \[
\int_t^{t+T} u_4(\tau) x_4(\tau) d\tau
\]
for several simulations. The average  value of the positive transition rates was kept the same in all simulations, and the initial conditions were always set to the steady state corresponding to the  constant inputs, that is,~$x(0)=e$. The transition rates were all of the form $u_i(t) = \mean{u}_i + A_i \cos(2\pi t + \phi_i)$, where $\phi_i \in [0,2\pi)$ and $A_i \in (0,\mean{u}_i)$ were chosen at  random. Thus, the rates are~$T$-periodic, with~$T=1$. 
It can be seen the average  production rate  in the periodic case  is always lower than the production rate for constant inputs, i.e.,~$R_P < R_C$.

Fig.~\ref{fig:zn} depicts $\gamma_4(t)$ over a single period in the  simulations. Solid lines show the instantaneous value and dashed lines show the average  over a single period. The black solid line is the value~$e_4$. Note that in some cases $\mean{\gamma}_4 > e_4$ while in others $\mean{\gamma}_4 < e_4$. In all cases $\gamma_4$ oscillates around~$e_4$, so~$z_4(t)=\gamma_4(t)-e_4$ changes sign.

The mean values of the inputs used in the simulations are $\mean{u}_1 = 13.56, \mean{u}_2 = 11.38, \mean{u}_3 = 3.90, \mean{u}_4 = 3.53$, and $\mean{u}_5 = 2.34$. The parameters $A_i,\phi_i$ of all inputs in each of the simulations are included in Table~\ref{tab:sim_params}.
\end{Example}

\begin{figure}
    \centering
    \includegraphics{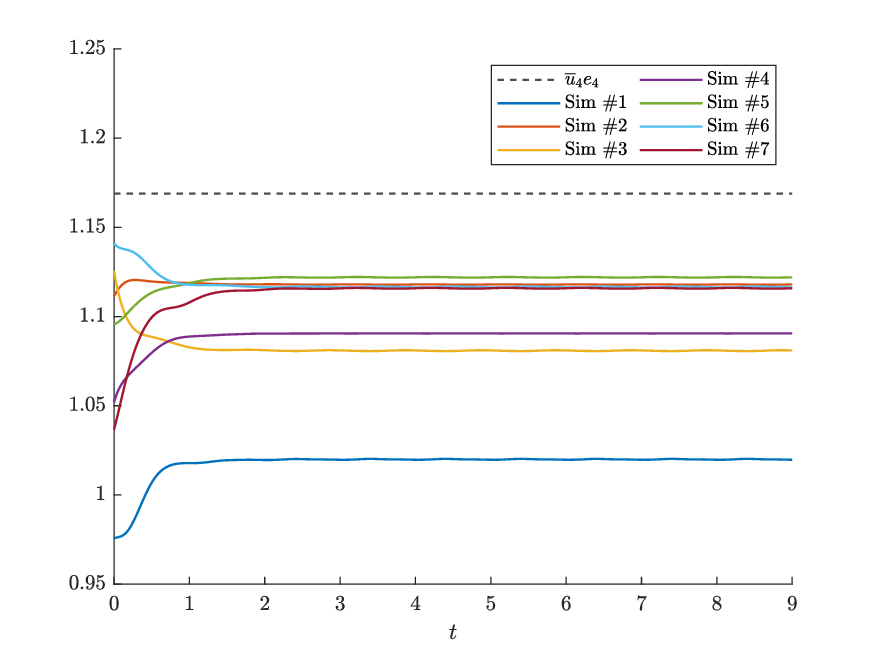}
    \caption{Moving average of the  production rate $u_4 x_4$ in an RFM with~$n=4$ for different choices of~$1 $-periodic  transition rates.}
    \label{fig:4d_unxn}
\end{figure}

\begin{figure}
    \centering
    \includegraphics{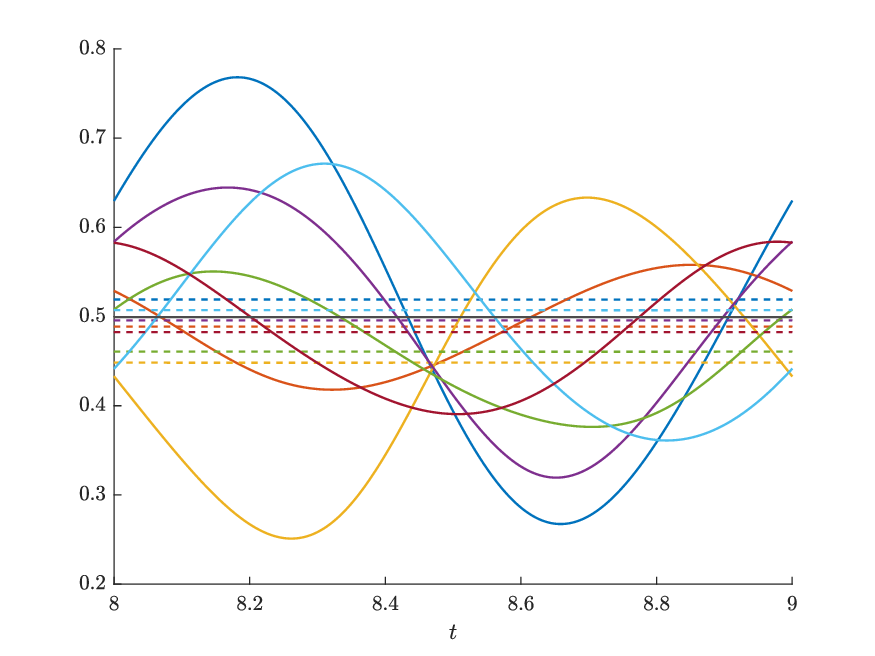}
    \caption{Plot of $\gamma_4(t)$ over a single period for several simulations using different transition rates with the same average values.}
    \label{fig:zn}
\end{figure}

\begin{table}
    \centering
    \begin{tabular}{|| c|| c c c c c c c c c c||}
    	\hline
   sim~\# &	$A_1$ & $\phi_1$ & $A_2$ & $\phi_2$ & $A_3$ & $\phi_3$ & $A_4$ & $\phi_4$ & $A_5$ & $\phi_5$ \\
    	\hline \hline
    1&	8.42 & 1.18 & 1.18 & 0.17 & 1.30 & 4.03 & 2.84 & 0.08 & 2.25 & 3.42 \\ 
    2&	6.78 & 3.06 & 4.87 & 5.77 & 3.72 & 3.24 & 0.78 & 1.61 & 1.11 & 1.30 \\ 
    3&	5.06 & 5.91 & 9.05 & 4.08 & 1.08 & 5.78 & 3.37 & 4.22 & 1.67 & 5.95 \\ 
4&    	9.43 & 0.76 & 0.88 & 1.73 & 3.46 & 4.87 & 3.15 & 6.16 & 0.87 & 4.11 \\ 
  5&  	6.54 & 2.02 & 3.23 & 5.92 & 3.14 & 5.57 & 2.40 & 1.09 & 1.19 & 1.64 \\ 
6&    	6.12 & 4.86 & 1.26 & 4.43 & 1.24 & 3.51 & 1.01 & 1.69 & 2.09 & 3.99 \\ 
  7&  	2.87 & 0.40 & 5.62 & 3.28 & 3.03 & 5.37 & 0.99 & 1.52 & 1.91 & 1.66 \\  
    	\hline 
    \end{tabular}\vspace*{0.25 cm}
    \caption{Parameters of the periodic inputs used in the seven 
    simulations in Example~\ref{exa:random}. Each row corresponds to a specific~simulation.}
    \label{tab:sim_params}
\end{table}
\section{Discussion}
We studied the GOE in the  $n$-dimensional RFM. Our analysis approach is new and is  based on calculating integrals of the (centered) state-variables  and ``higher-order moments'' of    these 
variables along the periodic solution. 

We proved that under certain conditions the $n$-dimensional RFM admits no GOE. We conjecture  that this in fact   always true, but even for the case~$n=2$ we are not able to prove this in the most general setting.

 An intuitive explanation for the fact that constant controls are always optimal, based on numerical  simulations,  is as follows. When~$u(t)$ is $T$-periodic and non-constant there are times~$t$ such that~$u(t)>\mean{u}$ and times when the opposite inequality holds. In general, the ``profit'' to the average production rate  in the 
 ``good'' times is less than the ``loss'' in the ``bad'' times.  An interesting line for further research is to try and rigorously formulate and prove this observation. Note that in the case of optimal periodic control for scalar   systems (i.e., when~$n=1$), convexity  theory
plays an important role~\cite{1dGOE,scalar_int}.

One case where we proved that there is no GOE is when~$n$ is odd and all 
the      internal rates are equal. 
It is interesting  to note that in~TASEP, with all rates {\sl constant}, the 
case of equal  internal hopping rates is the one that is amenable to analysis. Note also that it was recently shown that~TASEP, just like the RFM, is a contractive system (in a suitable stochastic sense)~\cite{TASEP_RANDOM_ATTRACTION}.

The RFM is a deterministic system, and it may be of interest to consider if the results on the~GOE 
hold also after adding some form of noise   to the system. 
Since  the RFM is contractive,
the results in~\cite{SLOTINE_RANDOM} suggest  that an  additive noise of small intensity (in the Ito-SDE-sense) will only slightly affect the output rate, which would imply that if the noise  does lead to  a GOE then it is likely to be small.  
In any  case,  a precise notion of  GOE in the presence of noise is, to the best of our knowledge, missing. 

Finally,  entrainment  to periodic excitations is important in many natural and artificial systems. 
  For example, connecting several artificial biological systems that entrain to a common clock may lead to a well-functioning modular system~\cite{control_syn_bio}. Epidemic outbreaks entrain to   seasonal forcing~\cite{seasonal_forcing}.
     Analysis of GOE in such models is related to interesting    questions like: (1) is periodic production of synthetic biology constructs more efficient than
constant production?   
and (2) does seasonality make the epidemics, on average,  more or less severe?

\subsection*{Acknowledgements} We thank the editor and the  anonymous reviewers  for a timely review process, and for many  detailed and helpful comments.  
\subsection*{Conflict of interests declaration} The authors declare no conflict of   interests.  
\subsection*{Author contributions} All authors performed the research and wrote the paper.


\end{document}